\documentclass[a4paper,12pt]{article}

\usepackage{amsmath,mathtools,amsthm,amssymb,dsfont,mathrsfs,color,graphicx}
\usepackage[toc,page]{appendix}
\usepackage{authblk,comment}

\newenvironment{Proof}{\vspace{1.5ex}{\sc Proof}.}{\vspace{2ex}}

\newtheorem{Theorem}{Theorem}
\newtheorem{Corollary}[Theorem]{Corollary}

\newtheorem{Lemma}[Theorem]{Lemma}
\newtheorem{Proposition}[Theorem]{Proposition}
\newtheorem{Remark}[Theorem]{Remark}

\usepackage{hyperref}

\hypersetup{
	colorlinks=true,
	linkcolor=blue,
	urlcolor=magenta,
	citecolor=red,
}

\DeclarePairedDelimiter\abs{\lvert}{\rvert}%
\DeclarePairedDelimiter\norm{\lVert}{\rVert}%

\makeatletter
\let\oldabs\abs
\def\abs{\@ifstar{\oldabs}{\oldabs*}}

\def\d{\displaystyle}

\def\e{\varepsilon}

\def\eop{\rule{0.7ex}{0.7ex}}
\def\eqd{: =}

\def\segoc#1#2{(#1, #2]}

\def\scal#1#2{
    \left\langle #1,
    #2\right\rangle
    }

\def\set#1#2{
    \left\{#1
    \left|\vphantom{#1#2}\right.
    #2\right\}
    }

\def\mod#1{
    |#1|
    }

\def\norm#1{
    \left\|
    #1\right\|
    }

\def\operator#1#2#3#4#5{
    \begin{array}{lcll}
    \d #1 \colon & \d #2 & \longrightarrow & \d #3 \\[1ex]
                 & \d #4 & \longmapsto     & \d #5
    \end{array}
    }

\def\ud{\,\mathrm{d}}

\def\supp{\mathop{\mathrm{supp}}}

\DeclareMathOperator{\argmin}{argmin}

\def\Ne{\mathds{N}}

\def\Re{\mathds{R}}

\def\1e{\mathds{1}}

\def\vgx{\mathbf{x}}

\def\vg0{\mathbf{0}}

\def\O{\mathscr{O}}
\def\P{\mathscr{P}}

\def\T{\mathscr{T}}

\newcommand{\ba}{\begin{equation}}
\newcommand{\ea}{\end{equation}}
\newcommand{\bean}{\begin{eqnarray}}
\newcommand{\eean}{\end{eqnarray}}
\newcommand{\bea}{\begin{eqnarray*}}
\newcommand{\eea}{\end{eqnarray*}}

\begin{document}

\title{\bf A mollifier approach to regularize
a Cauchy problem for the inhomogeneous Helmholtz equation}

\author{Pierre MAR\'ECHAL\footnote{
Institut de Mathématiques université Paul Sabatier, 31062 Toulouse, France. Email: pr.marechal@gmail.com}, \;  Walter Cedric SIMO TAO LEE\footnote{
Institut de Mathématiques université Paul Sabatier, 31062 Toulouse, France. Email: wsimotao@gmail.com}, \; Faouzi TRIKI\footnote{Laboratoire Jean Kuntzmann, UMR CNRS 5224, 
Universit\'{e} Grenoble-Alpes, 700 Avenue Centrale,
38401 Saint-Martin-d'H\`eres, France. Email: faouzi.triki@univ-grenoble-alpes.fr.  \\
{\sl   This work  is supported in part by the grant ANR-17-CE40-0029 of the French National Research Agency ANR (project MultiOnde).}} }

\maketitle

\begin{abstract}
The Cauchy problem for the inhomogeneous Helmholtz equation with
non-uniform refraction index is considered. The ill-posedness of
this problem is tackled by means of the variational form of
mollification. This approach is proved to be consistent, and
the proposed numerical simulations are quite promising.
\end{abstract}

\section{Introduction}
\label{introduction}
Let $V$ be a $C^{3,1}$ bounded domain of $\Re^3$ with boundary $\partial V$. For $\vgx^\prime \in \partial V$, we denote by   
 $\nu(\vgx^\prime)$  the unit normal vector to $\partial V$ pointing outward $V$. 
 Let $\Gamma$ be a nonempty open subset of $\partial V$. 
We consider the Cauchy problem for the inhomogeneous Helmholtz equation
\begin{align}
\Delta u(\vgx)+k^2\eta(\vgx)  u(\vgx)
&=S(\vgx), &\vgx\in V,\label{pt1}\\
\partial_\nu u(\vgx^\prime)&=f(\vgx^\prime), &\vgx^\prime \in\Gamma,\label{pt2}\\
u(\vgx^\prime)&=g(\vgx^\prime), &\vgx^\prime\in\Gamma.\label{pt3}
\end{align}

Here, $u=u(\vgx)$ is the unknown amplitude of
the incident field,
$\eta\in L^\infty(\Omega)$ is the refraction index, $k$ is a positive wave number,
$S\in L^2(\Omega)$ is the source function,
and $f\in L^2(\Gamma)$ and $g\in L^2(\Gamma)$ are empirically known boundary conditions.\\

The Helmholtz equation arises in a large range of applications related to  
the propagation  of acoustic and electromagnetic waves in the time-harmonic 
regime. In this paper, we consider the inverse problem 
of reconstructing an acoustic or electromagnetic field from partial data given  
on an open part 
of the boundary of a given domain.  This problem called the 
Cauchy problem for the Helmholtz equation is known to be ill-posed  if $\Gamma$ does not occupy the 
whole boundary $\partial V$ \cite{hadamard2003lectures, choulli2016applications, isakov2006inverse, alessandrini2009stability}. In~\cite{hong2018three}, the above system was considered in  the particular case where the refraction index is constant. However, in practice, the emitted wave travels through an environment in which the refraction index fails to be constant, and we need to investigate the corresponding problem.\\

We are facing a linear inverse problem. Our aim is to derive a
stable approximation method for this problem,
which yields stable and amenable computational scheme. 
Our main focus will be on {\sl mollification}, in the variational
sense of the term, which turns out to be both
flexible and numerically efficient.

Mollifiers were introduced in partial differential equations by
K.O. Fried\-richs~\cite{wiki:mollifier,friedrichs1944identity}. 
The term {\sl mollification} has been used in the field
of inverse problems since the eighties.
In the original works on the subject, mollifiers were
used to smooth the data prior to inversion.
In his book, D.A. Murio~\cite{murio2011mollification}
provides an overview of this
approach and its application to some classical inverse problems. 
Let us also mention the paper by
D.N. Hao~\cite{h1994mollification}, which
provides a wide framework for the mollification
approach in this initial meaning.
In~\cite{louis1990mollifier}, A.K. Louis and P. Maass
proposed another approach, based on
{\sl inner product duality}. This approach
has been subsequently referred to as {\sl the method
of approximate inverses}~\cite{schuster2007method}.
The approximate inverses are particularly well
adapted to problems in which
the adjoint equation has explicit solutions. A third
approach, based on a variational formulation,
also appeared in the same period of time.
In~\cite{lannes1987stabilized}, A. Lannes {\it et al.}
gave such a formulation while studying the problems
of Fourier extrapolation and deconvolution. 
This variational formulation was not studied further until the
papers by N. Alibaud {\it et al.}~\cite{alibaud2009variational}
and by X. Bonnefond and
P. Mar\'echal~\cite{bonnefond2009variational},
where convergence properties of the variational
formulation was considered.

A definite advantage of the variational approach to
mollification lies in the fact that it offers a quite
flexible framework, just like the Tikhonov regularization,
while being more respectful of the initial model equation
than the latter.

The paper is organized as follows. In Section~\ref{regularization},
we introduce the linear operator associated to the Cauchy problem
for the inhomogeneous Helmholtz equation.
We propose a regularized variational formulation of the
ill-posed problem based on mollification. Under an additional smoothness assumption
on the targeted solution we show in Theorem~\ref{maintheorem} that 
the unique minimizer converges strongly to the minimum-norm least square solution.
Section~\ref{discretization} is devoted to numerical experiments. 
We consider two numerical examples in order to illustrate the
efficiency of our regularization approach.    

\section{Functional setting and regularization}
\label{regularization}

We shall work in a functional space which enables us to interpret
the ideal (exact) data as the image of~$u$ by a bounded linear
operator. We observe that:

\begin{itemize}
\item[(I)]
the Laplacian~$\Delta$ is a bounded operator from $H^2(V)$ to $L^2(V)$,
so that, since $\eta\in L^\infty(\Omega)$, the operator $T_1 u=(\Delta+k^2 \eta)u$ is also
bounded from $H^2(V)$ to $L^2(V)$ \cite{gilbarg2015elliptic};
\item[(II)]
$\nabla u\in \left( H^1(V)\right)^3$, so that
$u\mapsto\partial_\nu u \in H^{1/2}(\Gamma)$
is a continuous linear operator, which implies in turn that the
operator $T_2 u=\partial_\nu u_{|\Gamma}$ is compact
from $H^2(V)$ to $L^2(\Gamma)$;
\item[(III)]
the trace operator $u\mapsto u_{|\partial V}$ maps $H^2(V)$
to $H^{3/2}(\Gamma)$ continuously, so that the
operator $T_3 u=u_{|\Gamma}$
is compact from $H^2(V)$ to $L^2(\Gamma)$.~\eop
\end{itemize}
Therefore, a natural choice for our workspace is $H^2(V)$.
We can then write our system in the form
\begin{equation}
\label{main equation}
Tu=v,    
\end{equation}
in which
$$
\operator{T}{H^2(V)}{L^2(V)\times L^2(\Gamma)\times L^2(\Gamma)}{u}%
{\big(\Delta u+k^2\eta u,\partial_\nu u,u\big)}
$$
and
$$
v=(S,f,g)\in G:= L^2(V)\times L^2(\Gamma)\times L^2(\Gamma).
$$
We first show that $T$ is injective. 

\begin{Proposition}\sf
\label{Pinjectivity}
The linear bounded map $T$ defined by \eqref{main equation}  is injective. 
\end{Proposition}

\begin{proof}
For a proof of this classical result, we refer to the Fritz John's
book~\cite{john1953partial}. More recent proofs
based on Carleman estimates can be found
in \cite{isakov2006inverse, choulli2016applications, alessandrini2009stability, triki2021h}.
The principal idea is to show that a part of~$\Gamma$ is non-characteristic
with respect to the Helmholtz operator, which in turns leads to the existence
of a small neighborhood of that part of the boundary in~$V$ where the solution
is identically zero. Since $\eta \in L^\infty(\Omega)$, the Helmholtz operator~$T_1$
possesses the unique continuation property in~$V$, and hence the solution
is identically zero in the whole domain~$V$ which completes the proof.  
\end{proof}

We are now going to set up our approach to the regularization
of the problem. We consider the mollifier
$$
\varphi_\alpha(\vgx)=
\frac{1}{\alpha^3}\varphi\left(\frac{\vgx}{\alpha}\right),
\quad\alpha\in\segoc{0}{1},\quad\vgx\in\Re^3,
$$
in which $\varphi$ is an integrable function such that
\begin{itemize}
    \item[(1)]
    $\supp{\varphi}\subset B_r\eqd\set{\vgx\in\Re^3}{\norm{\vgx}\leq r}$
    for some positive~$r$,
    \item[(2)]
    $\int\varphi(\vgx)\ud\vgx=1$.
\end{itemize}
Desirable additional properties of $\varphi$ are, as usual,
nonnegativity, isotropy, smoothness, radial decrease.

We denote by $C_\alpha$ the convolution operator by $\varphi_\alpha$:
for every $u\in L^2(\Re^3)$,
$C_\alpha u\eqd\varphi_\alpha\ast u$.
Our regularization principle will control smoothness by means
of~$C_\alpha$, and $\alpha$ will play the role of the
regularization parameter. 
One difficulty lies in the fact that convolving $u$ by $\varphi_\alpha$
entails extrapolating~$u$ from $V$ to the larger
set $V+\alpha B_r$. Zero padding is obviously forbidden here since
we wish to preserve $H^2$ regularity. It is then necessary to introduce an extension operator that preserves the properties 
of the solution on $ V$.

For $s\in [0,1)$,
we denote by $H^{2+s}_0(\Re^3)$ the set of functions in $H^{2+s}(\Re^3)$ 
having a compact support. Our objective is to derive an extension
operator $E\colon H^{2+s}(V)\to H^{2+s}_0(\Re^3)$ 
that in addition of satisfying  $Eu|_{V}=u$ for all $u\in H^{2+s}(V)$,
is bounded and invertible. Note that there are many extension operators
satisfying these properties. We next provide a complete characterization
of a useful extension operator suited for $C^{3,1}$ smooth bounded domains.
For other extension operators on Sobolev spaces under weaker regularity
assumptions on the domain $V,$ see, for example,
\cite{fefferman2014sobolev, maz1985continuity, calderon1961lebesgue, stein1970singular}.

For $\varepsilon\in (0, 1)$ small enough define the tubular domains 
\bean \label{tubularD}
V_\varepsilon^\pm:=\left\{\vgx^\prime\pm t\nu(\vgx^\prime): \; (\vgx^\prime, t)\in \partial V\times 
(0, \varepsilon) \right\}, \textrm{ and } V_\varepsilon := V \cup \partial V \cup 
V_\varepsilon^+.
\eean
We first notice that $V_\varepsilon^- \subset V \subset V_\varepsilon$
for all $\varepsilon\in (0, 1)$. Due to the regularity of $\partial V$,
the function defined by
\bean \label{phi}
\phi(\vgx^\prime, t)= \vgx^\prime-t\nu(\vgx^\prime),
\eean
is a $C^{2,1}$-diffeomorphism from $\partial V\times (0, \varepsilon)$
onto $V_{\varepsilon}^-$ for $\varepsilon$ small enough.

Let $u$ be fixed 
in $H^2(V)$. The first step is to construct an extension $u_\varepsilon$
of~$u$ to $H^2(V_{\varepsilon/2})$. For $\alpha,\beta\in\Re$, set 
\bean
\tilde u(\vgx^\prime+t\nu(\vgx^\prime)) =
\alpha u(\vgx^\prime-t \nu(\vgx^\prime))+ \beta u(\vgx^\prime-2 t \nu(\vgx^\prime)), 
\textrm{ for } 
(\vgx^\prime,t) \in \partial V \times (0, \varepsilon/2). 
\eean
Since $\nu$ is a $C^2$ vector field, $\tilde u$ lies in $H^2(V_\varepsilon^+)$, and verifies
\bean \label{jumps}
\tilde u(\vgx^\prime) = (\alpha+\beta)u(\vgx^\prime), \textrm{ and } \;
\partial_\nu \tilde u(\vgx^\prime)
=(-\alpha-2\beta)\partial_\nu u(\vgx^\prime), \;\; \vgx^\prime \in \partial V.
\eean
Let $u_\varepsilon$ be defined by
\bean \label{uva}
u_\varepsilon(\vgx) = \left\{  \begin{array}{llcc}
\tilde u(\vgx),  &\vgx\in V_{\varepsilon/2}^+,\\
u(\vgx), & \vgx\in V.
\end{array} 
\right.
\eean
By construction, we have $u_\varepsilon \in H^2(V)\cup H^2(V_\varepsilon^+)$.
Considering the traces \eqref{jumps} and taking $\alpha =3$ and $\beta=-2$,
implies that $u_\varepsilon$ and its first derivatives have no jumps
across $\partial V$, and thus  $u_\varepsilon \in H^2(V_{\varepsilon/2})$.

Let $ \chi_\varepsilon\in C^\infty_0(\Re^3)$ be a cut off function satisfying 
\bean \label{chi}
 \chi_\varepsilon= 1, \textrm{ on } V_{\varepsilon/8},
\textrm{  and  }  \chi_\varepsilon = 0,  \textrm{ on } \Re^3
\setminus \overline{V_{3\varepsilon/8}}.
\eean 
Now, we are ready to introduce the operator~$E$. For $u\in H^2(V)$, define 
\bea
Eu = \chi_\varepsilon u_\varepsilon,
\eea
where $u_\varepsilon$ and $\chi_\varepsilon$ are respectively given
in \eqref{uva} and \eqref{chi}.

\begin{Proposition}\sf
Let $s\in [0,1)$ be fixed. 
The extension operator  $E\colon H^{2+s}(V)\to H^{2+s}_0(\Re^3)$ is bounded,
invertible, and satisfies 
\bean
 \norm{u}_{H^{2+s}(V)} \leq \norm{Eu}_{H^{2+s}(\Re^3)}
\leq C_s \norm{u}_{H^{2+s}(V)},
\eean
where $C_s >1$ is a constant that only depends on $s$,
$\varepsilon$ and $V$. In addition, $\textrm{Supp}(Eu) \subset
V_\varepsilon$.  
\end{Proposition}

\begin{proof}
The left side inequality is straightforward. 
The functions $\psi_{j} : V_{\varepsilon/j}^+ \to  V_{\varepsilon/j}^-, j=1, 2$, defined by 
\bean \label{psi}
\psi_{j}(\vgx^\prime+ t\nu(\vgx^\prime))= \vgx^\prime-jt\nu(\vgx^\prime),\quad (x^\prime, t)\in
\partial V\times (0, \varepsilon),\; j=1, 2, 
\eean
are  $C^{2,1}$-diffeomorphisms.

Forward calculations give
\bea
\int_{\Re^3 \setminus\overline V}|Eu|^2+|\nabla Eu|^2
+|\nabla^2 Eu|^2\ud\vgx \\
\leq  \|\chi_\varepsilon\|_{C^{2,1}(\Re^3)}^2 \sum_{j=1}^2
\int_{V_{\varepsilon/2}^+}  |\tilde u|^2 +|\nabla \tilde u|^2 +|\nabla^2 \tilde u|^2 \ud\vgx\\
\leq \kappa\|\chi_\varepsilon\|_{C^{2,1}(\Re^3)}^2\sum_{j=1}^2 
\left(\|\psi_{j}\|_{C^{2,1}( V^+_{\varepsilon/j})}^2 +1\right)
\|D\psi^{-1}_{j}\|_{C^1( V^-_{\varepsilon/j})}^2
\int_{V_{\varepsilon}^-} |\nabla u|^2+ |\nabla^2 u|^2 \ud\vgx,
\eea
with $D\psi^{-1}_{j}$ is the gradient of the vector
field $\psi^{-1}_{j}$, and $\kappa>0$ is a universal constant. 
Therefore 
\bea
\norm{E u}_{H^2(\Re^3)}
\leq C \norm{u}_{H^2(V)},
\eea
where 
\bea
C^2 = \kappa  \|\chi_\varepsilon\|_{C^{2,1}(\Re^3)}^2\sum_{j=1}^2 
\left(\|\psi_{j}\|_{C^{2,1}( V^+_{\varepsilon/j})}^2 +1\right)
\|D\psi^{-1}_{j}\|_{C^1( V^-_{\varepsilon/j})}^2. 
\eea
Similarly, tedious calculations lead to the following estimate of the seminorm 
\bea
|D_\beta Eu|_{s, \Re^3}^2=\int_{\Re^3}\int_{\Re^3}\frac{\left|D_\beta Eu(x)-D_\beta Eu(y) \right|^2}{|x-y|^{3+2s}}\ud x\ud y
\leq\widetilde C_s^2(|D_\beta u|_{s,V}^2 +
\norm{u}_{H^2(V)}^2),
\eea
for all $\beta\in \mathbb N^3$, satisfying $|\beta| = 2$, where $ \widetilde C_s>0$, depends on
$s, \varepsilon$ and $V$. By taking $C_s= \max(1, C, \widetilde C_s)$, we complete the proof 
of the proposition.  
\end{proof}

The  defined extension operator~$E$ opens the way to
the following variational formulation of mollification:
$$
(\P)\quad\left|
\begin{array}{rl}
\hbox{Minimize}&
J_\alpha(u,v,T)\eqd\norm{v-Tu}_G^2+\norm{(I-C_\alpha)Eu}_{H^2(\Re^3)}^2\\
\hbox{subject to}&
u\in H^2(V).
\end{array}
\right.
$$
Our aim is now to prove
\begin{enumerate}
\item
the well-posedness of the above variational
problem, that is, that the solution $u_\alpha$ 
depends continuously on the data~$v$;
\item
the consistency of the regularization, that is, that
$u_\alpha$ converges to $T^\dagger u$ in some sense
as $\alpha\downarrow 0$, where  $T^\dagger $ is the 
pseudo-inverse of $T$. 
\end{enumerate}

\begin{Lemma}\sf
\label{from-alibaud-et-al}
Let $K,s>0$.
Let $\varphi\in L^1(\Re^3)$ be such that
$$
\hat\varphi(0)=1
\quad\hbox{and}\quad
1-\hat\varphi(\xi)\sim K\norm{\xi}^s
\quad\hbox{as}\quad
\xi\to 0.
$$
Assume in addition that $\hat\varphi(\xi)\not=1$ for every $\xi\in\Re^3\setminus\{0\}$. 
Define
$$
m_\alpha=\min_{\norm{\xi}=1}\mod{1-\hat\varphi(\alpha\xi)}^2
\quad\hbox{and}\quad
M_\alpha=\max_{\norm{\xi}=1}\mod{1-\hat\varphi(\alpha\xi)}^2.
$$
Then,
\begin{itemize}
\item[(i)]
for every $\alpha>0$, $0<m_\alpha\leq M_\alpha\leq(1+\norm{\varphi}_1)^2$;
\item[(ii)]
$\sup_{\alpha>0}M_\alpha/m_\alpha<\infty$ and
$M_\alpha\to 0$ as $\alpha\downarrow 0$;
\item[(iii)]
there exists $\nu_\circ,A_\circ>0$ such that, for every $\alpha\in\segoc{0}{1}$,
for every $\xi\in\Re^3\setminus\{0\}$,
$$
\nu_\circ\left(\norm{\xi}^{2s}\1e_{B_{1/\alpha}}(\xi)+
\frac{1}{M_\alpha}\1e_{B_{1/\alpha}^c}(\xi)\right)
\leq
\frac{\mod{1-\hat\varphi(\alpha\xi)}^2}{\mod{1-\hat\varphi(\alpha\xi/\norm{\xi})}^2}
\leq
\mu_\circ\norm{\xi}^{2s}.
$$
\end{itemize}
\end{Lemma}

\begin{Proof}
See \cite[Lemma 12]{alibaud2009variational}.~\eop
\end{Proof}

\begin{Corollary}\sf
\label{corollary-to-alibaud-et-al}
Let $\varphi$, $m_\alpha$, $M_\alpha$, $\nu_\circ$, $\mu_\circ$ be
as in Lemma~\ref{from-alibaud-et-al}, and let
$C_\alpha$ be the operator of convolution with $\varphi_\alpha$,
For every $u\in L^2(\Re^3)$,
\begin{equation}
\label{minoration-I-C-u-L2}
\norm{(I-C_\alpha)u}_{L^2(\Re^3)}^2\geq
\nu_\circ m_\alpha
\int_{\Re^3}\left(\norm{\xi}^{2s}\1e_{B_{1/\alpha}}(\xi)+
\frac{1}{M_\alpha}\1e_{B_{1/\alpha}^c}(\xi)\right)\mod{\hat{u}(\xi)}^2\ud\xi.
\end{equation}
\end{Corollary}

\begin{Proof}
From Lemma~\ref{from-alibaud-et-al}, we have:
\begin{eqnarray*}
\norm{(I-C_\alpha)u}_{L^2(\Re^3)}^2
&=&
\int_{\Re^3}\mod{1-\hat{\varphi}(\alpha\xi)}^2\mod{\hat{u}(\xi)}^2\ud\xi\\
&=&
\int_{\Re^3}\mod{1-\hat{\varphi}(\alpha\xi/\norm{\xi})}^2
\frac{\mod{1-\hat{\varphi}(\alpha\xi)}^2}{\mod{1-\hat{\varphi}(\alpha\xi/\norm{\xi})}^2}
\mod{\hat{u}(\xi)}^2\ud\xi\\
&\geq&
m_\alpha\nu_\circ
\int_{\Re^3}\left(\norm{\xi}^{2s}\1e_{B_{1/\alpha}}(\xi)+
\frac{1}{M_\alpha}\1e_{B_{1/\alpha}^c}(\xi)\right)\mod{\hat{u}(\xi)}^2\ud\xi.
\end{eqnarray*}
\end{Proof}

\begin{Lemma}\sf
\label{minoration-(I-C)u}
Let $C_\alpha$ be the operator of convolution with $\varphi_\alpha$,
where $\varphi$ is as in Lemma~\ref{from-alibaud-et-al}.
There exists a positive constant $A_\circ$, depending on $V$ only, such that
for every $s\geq 0$ and every $u\in H^s(V)$,
\begin{equation}
\label{eq-minoration-(I-C)u}
A_\circ m_\alpha\norm{u}_{H^s(\Re^3)}^2\leq
\norm{(I-C_\alpha)u}_{H^s(\Re^3)}^2.
\end{equation}
\end{Lemma}

\begin{Proof}
We start with the case $s=0$. 
From~\eqref{minoration-I-C-u-L2}, we have:
\begin{eqnarray*}
\norm{(I-C_\alpha)u}_{L^2(\Re^3)}^2
&\geq&
m_\alpha \nu_\circ
\int_{\Re^3}\left(\norm{\xi}^{2s}\1e_{B_{1/\alpha}\setminus B_1}(\xi)+
\frac{1}{M_\alpha}\1e_{B_{1/\alpha}^c}(\xi)\right)\mod{\hat{u}(\xi)}^2\ud\xi\\
&\geq&
m_\alpha \nu_\circ
(1+\norm{\varphi}_1)^{-2}
\int_{\Re^3}\1e_{B_1^c}(\xi)\mod{\hat{u}(\xi)}^2\ud\xi\\
&\geq&
m_\alpha \nu_\circ
(1+\norm{\varphi}_1)^{-2}
\norm{T_{B_1^c}^{-1}}^2\norm{u}_{L^2(\Re^3)}^2,
\end{eqnarray*}
in which $T_{B_1^c}\eqd\1e_{B_1^c}F$, the operator of Fourier truncation to $B_1^c$.
Thus
$$
A_\circ=\nu_\circ(1+\norm{\varphi}_1)^{-2}\norm{T_{B_1^c}^{-1}}^2
$$
is suitable. Notice that, for $s=0$, Parseval's identity
enables to rewrite~ \eqref{eq-minoration-(I-C)u} in the form
\begin{equation}
\label{eq-minoration-(I-C)u-bis}
A_\circ m_\alpha\norm{\hat u}_{L^2(\Re^3)}^2\leq
\norm{(1-\hat\varphi(\alpha\cdot)\hat u}_{L^2(\Re^3)}^2,
\end{equation}
in which $F$ denotes the Fourier-Plancherel operator.
Now, let $s>0$ and assume that $u\in H^s(V)$.
We readily see that
$F^{-1}\big((1+\norm{\xi}^2)^{s/2}\hat u(\xi)\big)$ belongs
to $L^2(\Re^3)$. Applying~\eqref{eq-minoration-(I-C)u-bis}
to the latter function yields
$$
A_\circ m_\alpha\norm{(1+\norm{\cdot}^2)^{s/2} \hat u}_{L^2(\Re^3)}^2\leq
\norm{(1-\hat\varphi(\alpha\cdot))(1+\norm{\cdot}^2)^{s/2} \hat u}_{L^2(\Re^3)}^2,
$$
and \eqref{eq-minoration-(I-C)u} follows.~\eop
\end{Proof}

\begin{Lemma}\sf
\label{majoration-(I-C)u}
Let $C_\alpha$ be as in the previous lemma,  $s>0$, and  $ s_\circ\geq 0$.
If $u\in H^{s_\circ+s}(\Re^3)$, then 
$$
\norm{(I-C_\alpha)u}_{H^{s_\circ}(\Re^3)}^2\leq
\mu_\circ M_\alpha\norm{u}_{H^{s_\circ+s}(\Re^3)}^2,
$$
with $\mu_\circ$ is the positive constant provided
by Lemma~\ref{from-alibaud-et-al}.
\end{Lemma}

\begin{Proof}
We have:
\begin{eqnarray*}
\norm{(I-C_\alpha)u}_{H^{s_\circ}(\Re^3)}^2
&=&
\int_{\Re^3}\mod{1-\hat{\varphi}(\alpha\xi/\norm{\xi})}^2
\frac{\mod{1-\hat{\varphi}(\alpha\xi)}^2}{\mod{1-\hat{\varphi}(\alpha\xi/\norm{\xi})}^2}
(1+\norm{\xi}^2)^{s_\circ}\mod{\hat{u}(\xi)}^2\ud\xi\\
&\leq&
\mu_\circ M_\alpha\int_{\Re^3}
\norm{\xi}^{2s}(1+\norm{\xi}^2)^{s_\circ}\mod{\hat{u}(\xi)}^2\ud\xi\\
&\leq&
\mu_\circ M_\alpha\int_{\Re^3}
(1+\norm{\xi}^2)^{s_\circ+s}\mod{\hat{u}(\xi)}^2\ud\xi\\
&=&
\mu_\circ M_\alpha
\norm{u}_{H^{s_\circ+s}(\Re^3)}^2,
\end{eqnarray*}
in which the first inequality stems from Lemma~\ref{from-alibaud-et-al}.~\eop
\end{Proof}

\begin{Theorem}\sf \label{maintheorem}
Assume $u^\dagger\in H^{2+s}(\Re^3)$, with $s\in (0,1)$, and  $v=Tu^\dagger_{|V}$, so that $u^\dagger_{|V}=T^\dagger v$.
Let $u_\alpha$ be the solution to Problem~$(\P)$.
Then $u_\alpha\to T^\dagger v$ in $H^2(V)$ as $\alpha\downarrow 0$.
\end{Theorem}

\begin{Proof}
We walk in the steps of the proof of Theorem 11 in~\cite{alibaud2009variational},
which we adapt to the present context. The main differences lie
in that the regularization term uses a Sobolev norm and in that we make use of the
extension operator~$E$ in order to cope with boundary constraints.
In Step~1, we show that the family $(u_\alpha)$ is bounded in $H^2(V)$,
thus weekly compact; in Step~2, we establish the weak convergence of $u_\alpha$
to $T^\dagger v$, and finally in Step~3, we use a compactness argument to show
that the convergence is, in fact, strong.

{\sl Step 1.}
By construction, we have:
\begin{equation}
\label{initial-estimate}
\norm{(I-C_\alpha)Eu_\alpha}_{H^2(\Re^3)}^2\leq
J_\alpha(u_\alpha,v,T)\leq
J_\alpha(u^\dagger_{|V},v,T)\leq
\norm{(I-C_\alpha)Eu^\dagger_{|V}}_{H^2(\Re^3)}^2.
\end{equation}
Using Lemma~\ref{minoration-(I-C)u} and Lemma~\ref{majoration-(I-C)u}, we obtain:
$$
\norm{u_\alpha}_{H^2(V)}^2\leq
\norm{Eu_\alpha}_{H^2(\Re^3)}^2\leq
\frac{\mu_\circ}{A_\circ}\frac{M_\alpha}{m_\alpha}\norm{Eu^\dagger_{|V}}_{H^{2+s}(\Re^3)}^2,
$$
and Lemma~\ref{from-alibaud-et-al}(ii) then shows that
the set $(u_\alpha)_{\alpha\in\segoc{0}{1}}$ is bounded
in $H^2(V)$, therefore is weakly compact.

{\sl Step 2.} Denote $\norm{\cdot}_G$ the natural norm
in the Hilbert space $G$. 
Now, let $(\alpha_n)_n$ be a sequence converging to $0$. There then exists
a subsequence $(u_{\alpha_{n_k}})_k$ which converges weakly in $H^2(V)$.
Let $\tilde{u}$ be the weak limit of this subsequence. We then have:
\begin{eqnarray*}
\norm{v-Tu_{\alpha_{n_k}}}_G^2
&\leq&
J_{\alpha_{n_k}}(u_{\alpha_{n_k}},v,T)\\
&\leq&
J_{\alpha_{n_k}}(u^\dagger_{|V},v,T)\\
&=&
\norm{(I - C_{\alpha_{n_k}})u^\dagger_{|V}}_{H^2(\Re^3)}^2\\
&\leq&
\mu_\circ M_{\alpha_{n_k}}\norm{u^\dagger_{|V}}_{H^{2+s}(\Re^3)}^2 .
\end{eqnarray*}
Since $M_{\alpha_{n_k}}$ goes to zero as $k\to\infty$, 
so does $\norm{v-Tu_{\alpha_{n_k}}}_G^2$.
By weak lower semicontinuity of the norm on the Hilbert space~$G$,
we see that the weak limit $\tilde{u}$ satisfies:
$$
\norm{v-T\tilde{u}}_G^2\leq
\liminf_{k\to\infty}\norm{v-Tu_{\alpha_{n_k}}}_G^2=
\lim_{k\to\infty}\norm{v-Tu_{\alpha_{n_k}}}_G^2=0.
$$
Therefore, $T\tilde{u}=v$ and the injectivity of $T$ (Proposition \ref{Pinjectivity})
implies that $\tilde{u} = u^\dagger$.\\
{\sl Step 3.}
We will show that, for every multi-index $\beta\in\Ne^3$
such that $\mod{\beta}\leq 2$,
\begin{equation}
\label{L2-convergence-DLu}
D_\beta Eu_\alpha\to D_\beta Eu^\dagger_{|V}
\hbox{ in }
L^2(\Re^3)
\hbox{ as }
\alpha\downarrow 0,
\end{equation}
which will imply the announced strong convergence.
Observe first that, by the previous step and the continuity
of~$E$, 
$$
Eu_\alpha\rightharpoonup Eu^\dagger_{|V}
\hbox{ in }
H^2(\Re^3)
\hbox{ as }
\alpha\downarrow 0,
$$
so that, for every multi-index $\beta\in\Ne^3$
such that $\mod{\beta}\leq 2$,
\begin{equation}
\label{Fréchet-Kolmo-1}
D_\beta Eu_\alpha\rightharpoonup D_\beta Eu^\dagger_{|V}
\hbox{ in }
L^2(\Re^3)
\hbox{ as }
\alpha\downarrow 0,
\end{equation}
Fix $\beta\in\set{\beta'\in\Ne^3}{\mod{\beta'}\leq 2}$ and
let $(\alpha_n)_n$ be a sequence converging to $0$,
as in the previous step. For convenience,
let $u_n\eqd u_{\alpha_n}$, $C_n\eqd C_{\alpha_n}$,
$M_n\eqd M_{\alpha_n}$ and $m_n\eqd m_{\alpha_n}$.
Since $Eu_n$ has compact support, it is obvious that
\begin{equation}
\label{Fréchet-Kolmo-2}
\lim_{R\to\infty}\sup_n\int_{\norm{x}\geq R}
\mod{D_\beta Eu_n(x)}^2\ud x=0.
\end{equation}
Now, for every $h\in\Re^3$ and every function~$u$,
let $\T_h u$ denote the translated function $x\mapsto u(x-h)$.
We proceed to show that
\begin{equation}
\label{Fréchet-Kolmo-3}
\sup_n\norm{\T_h D_\beta Eu_n-D_\beta Eu_n}_{L^2(\Re^3)}^2
\to 0
\hbox{ as }
\norm{h}\to 0.
\end{equation}
Together with~\eqref{Fréchet-Kolmo-1} and~\eqref{Fréchet-Kolmo-2},
this will establish~\eqref{L2-convergence-DLu} via
the Fréchet-Kolmogorov Theorem
(see {\it e.g.}~\cite[Theorem 3.8 page 175]{hirsch2012elements}).
We have:
\begin{eqnarray*}
\norm{\T_h D_\beta Eu_n-D_\beta Eu_n}_{L^2(\Re^3)}^2
&=&
\norm{F(\T_h D_\beta Eu_n-D_\beta Eu_n)}_{L^2(\Re^3)}^2\\
&=&
\int\mod{e^{-2i\pi\scal{h}{\xi}}-1}^2
\mod{FD_\beta Eu_n(\xi)}^2\ud\xi\\
&=&
I_1+I_2,
\end{eqnarray*}
in which
\begin{eqnarray*}
I_1
&\eqd&
\int_{\norm{\xi}\leq 1/\alpha_n}\mod{e^{-2i\pi\scal{h}{\xi}}-1}^2
\mod{FD_\beta Eu_n(\xi)}^2\ud\xi,\\
I_2
&\eqd&
\int_{\norm{\xi}> 1/\alpha_n}\mod{e^{-2i\pi\scal{h}{\xi}}-1}^2
\mod{FD_\beta Eu_n(\xi)}^2\ud\xi.
\end{eqnarray*}
We now bound $I_1$ and $I_2$.
On the one hand,
\begin{eqnarray*}
I_1
&=&
\int_{0<\norm{\xi}\leq 1/\alpha_n}
\frac{\mod{e^{-2i\pi\scal{h}{\xi}}-1}^2}{\norm{\xi}^{2s}}
\norm{\xi}^{2s}\mod{FD_\beta Eu_n(\xi)}^2\ud\xi\\
&\leq&
\sup_{\xi\not=0}
\frac{\mod{e^{-2i\pi\scal{h}{\xi}}-1}^2}{\norm{\xi}^{2s}}
\int_{\norm{\xi}\leq 1/\alpha_n}
\norm{\xi}^{2s}\mod{FD_\beta Eu_n(\xi)}^2\ud\xi.
\end{eqnarray*}
Since $\mod{e^{-2i\pi\scal{h}{\xi}}-1}=\O(\norm{\xi})$ and $s\leq 1$,
the above supremum is finite. Let $\gamma_\circ$ denote its value.
Therefore
\begin{eqnarray*}
I_1
&\leq&
\gamma_\circ\norm{h}^{2s}
\left(
\int_{\norm{\xi}\leq 1/\alpha_n}\mod{FD_\beta Eu_n(\xi)}^2\ud\xi+
\int_{\norm{\xi}\leq 1/\alpha_n}\norm{\xi}^{2s}\mod{FD_\beta Eu_n(\xi)}^2\ud\xi
\right)\\
&\leq&
\gamma_\circ\norm{h}^{2s}
\left(
\norm{FD_\beta Eu_n}_{L^2(\Re^3)}^2+
\int_{\norm{\xi}\leq 1/\alpha_n}\norm{\xi}^{2s}\mod{FD_\beta Eu_n(\xi)}^2\ud\xi\right).
\end{eqnarray*}
Since $(u_n)$ is bounded in $H^2(V)$ independently
of~$n$, so is
$(\norm{FD_\beta Eu_n}_{L^2(\Re^3)}^2)$. Moreover,
by using Corollary~\ref{corollary-to-alibaud-et-al}
with $D_\beta Eu_n$ in place of~$u$, we get
\begin{eqnarray*}
\int_{\norm{\xi}\leq 1/\alpha_n}\norm{\xi}^{2s}\mod{FD_\beta Lu_n(\xi)}^2\ud\xi
&\leq&
\frac{1}{\nu_\circ m_n}
\norm{(I-C_n)D_\beta Eu_n}_{L^2(\Re^3)}^2\\
&=&
\frac{1}{\nu_\circ m_n}
\norm{D_\beta(I-C_n)Eu_n}_{L^2(\Re^3)}^2\\
&\leq&
\frac{1}{\nu_\circ m_n}
\norm{(I-C_n)Eu_n}_{H^2(\Re^3)}^2\\
&\leq&
\frac{1}{\nu_\circ m_n}
\norm{(I-C_n)Eu_{|V}^\dagger}_{H^2(\Re^3)}^2\\
&\leq&
\frac{\mu_\circ}{\nu_\circ}\frac{M_n}{m_n}
\norm{Eu_{|V}^\dagger}_{H^{2+s}(\Re^3)}^2,
\end{eqnarray*}
in which the last two inequalities are respectively due to
the inequality~\eqref{initial-estimate} and
Lemma~\ref{majoration-(I-C)u} (with $s_\circ=2$).
It follows that $I_1=\O(\norm{h}^{2s})$.
On the other hand, using again Corollary~\ref{corollary-to-alibaud-et-al} 
with $D_\beta Eu_n$ in place of~$u$, we have:
\begin{eqnarray*}
I_2
&\leq&
4\int_{\norm{\xi}>1/\alpha_n}\mod{FD_\beta Eu_n(\xi)}^2\ud\xi\\
&\leq&
\frac{4}{\nu_\circ}\frac{M_n}{m_n}
\norm{(I-C_n)D_\beta Eu_n}_{L^2(\Re^3)}^2\\
&=&
\frac{4}{\nu_\circ}\frac{M_n}{m_n}
\norm{D_\beta(I-C_n)Eu_n}_{L^2(\Re^3)}^2\\
&\leq&
\frac{4}{\nu_\circ}\frac{M_n}{m_n}
\norm{(I-C_n)Eu_n}_{H^2(\Re^3)}^2\\
&\leq&
\frac{4}{\nu_\circ}\frac{M_n}{m_n}
\norm{(I-C_n)Eu_{|V}^\dagger}_{H^2(\Re^3)}^2\\
&\leq&
4\frac{\mu_\circ}{\nu_\circ}\frac{M_n}{m_n}M_n
\norm{Eu_{|V}^\dagger}_{H^2(\Re^3)}^{2+s},
\end{eqnarray*}
in which the last two inequalities are respectively due to
the inequality~\eqref{initial-estimate} and
Lemma~\ref{majoration-(I-C)u} (with $s_\circ=2$).
It follows that $I_2=\O(M_n)$.
Gathering the obtained bounds on $I_1$ and $I_2$,
we see that there exists a positive constant~$K$ such that
\begin{equation}
\label{estimate-from-I2-I2}
\norm{\T_h D_\beta Eu_n-D_\beta Eu_n}_{L^2(\Re^3)}^2\leq
K\big(\norm{h}^{2s}+M_n\big).
\end{equation}
Now, fix $\e>0$. There exists $n_\circ\in\Ne^*$ such that
for every $n\geq n_\circ$, $M_n\leq\e$.
From~\eqref{estimate-from-I2-I2}, we see that
\begin{eqnarray*}
\lefteqn{\sup_n\norm{\T_h D_\beta Eu_n-D_\beta Eu_n}_{L^2(\Re^3)}^2}\\
&\leq&
\max\big\{
\max_{1\leq n\leq n_\circ}\norm{\T_h D_\beta Eu_n-D_\beta Eu_n}_{L^2(\Re^3)}^2,
K\big(\norm{h}^{2s}+\e\big).
\big\}
\end{eqnarray*}
By the $L^2$-continuity of translation, we have:
$$
\forall n\in\Ne^*,\quad
\norm{\T_h D_\beta Eu_n-D_\beta Eu_n}_{L^2(\Re^3)}^2\to 0
\quad\hbox{as}\quad
\norm{h}\to 0.
$$
Consequently,
$$
\max_{1\leq n\leq n_\circ}\norm{\T_h D_\beta Eu_n-D_\beta Eu_n}_{L^2(\Re^3)}^2
\to 0
\quad\hbox{as}\quad
\norm{h}\to 0,
$$
so that
$$
\limsup_{h\to 0}\sup_n
\norm{\T_h D_\beta Eu_n-D_\beta Eu_n}_{L^2(\Re^3)}^2
\leq K\e.
$$
Since $\e>0$ was arbitrary, \eqref{Fréchet-Kolmo-3}
is established, which achieves the proof.~\eop
\end{Proof}

\begin{Remark}\sf
Notice that the real $s>0$ in Theorem~\ref{maintheorem} can be
taken arbitrary large. The proof is similar, we only need to consider
an extension operator~$E$ that is bounded from $H^{2+s}(V)$
into $H^{2+s}_0(\Re^3)$.
\end{Remark}

\section{Numerical experiments}
\label{discretization}

In this section, we consider two numerical examples
in order to illustrate the accuracy and the efficiency
of our regularization approach in the resolution of the
in-homogeneous Helmholtz equation with non-constant refraction index.

In order to reduce the computational complexity, we consider the system \eqref{pt1}-\eqref{pt2}-\eqref{pt3}
in two dimensions (as in~\cite{hieu2016regularization}) on a rectangular domain as follows:
\begin{align}
u_{xx}(x,y) + u_{yy}(x,y) +k^{2}\eta(x,y)u(x,y)&=S(x,y), &(x,y)\in[a,b]\times[0,1],\label{pt12d}\\
u_{y}(x,0)&=f(x), & x \in[a,b],\label{pt22d}\\
u(x,0)&=g(x), & x \in[a,b].\label{pt32d}
\end{align}
Given the boundary data $f$ and $g$ at $y=0$, we aim at approximating the solution $u(\cdot,y)$ for $y \in(0,1]$.\\

\textbf{Example 1:} For the first example, we set $[a,b] = [-1,1]$, $k =3$ and define the refraction index $\eta$ and the exact solution $u$ as:

\begin{equation*}
\label{def refract index exple 1}
\eta(x,y) = \begin{cases}
2- (x^2 + \frac{(2y-1)^2}{0.8^2})^{1/2} & \text{if} \,\,  x^2 + \frac{(2y-1)^2}{0.8^2} \leq 1,\\
1 & \text{otherwise}
\end{cases}
\end{equation*}
and
\begin{equation*}
    \label{def sol u exple 1}
    u(x,y) = (x-2y+1)\sin{\left(\frac{k}{\sqrt{2}}\left(x+2y-1\right)\right)}.
\end{equation*}
The source term $S$ and the boundary data $f$ and $g$ are defined accordingly:
\begin{equation*}
\begin{cases}
    S(x,y) &= -\frac{5}{2}k^2(x-2y+1)\sin{\left(\frac{k}{\sqrt{2}}\left(x+2y-1\right)\right)} - 3 k \sqrt{2} \cos{\left(\frac{k}{\sqrt{2}}\left(x+2y-1\right)\right)} \\ 
    & \quad + \,\,k^2\eta(x,y)u(x,y),\\
    f(x) &= k \sqrt{2}(x+1) \cos{\left(\frac{k}{\sqrt{2}}\left(x-1\right)\right)} - 2 \sin{\left(\frac{k}{\sqrt{2}}\left(x-1\right)\right)},\\
    g(x) &= (x+1)\sin{\left(\frac{k}{\sqrt{2}}\left(x-1\right)\right)}.
    \end{cases}
\end{equation*}

\textbf{Example 2:} For the second example, we consider a simpler setting where $[a,b] = [-1.5,1.5]$, $k =1$, and define the refraction index $\eta$ (depending only on $y$) and the exact solution $u$ as:

\begin{equation*}
\label{def refract index exple }
\eta(x,y) = 1+ y^2,
\end{equation*}
and
\begin{equation*}
    \label{def sol u exple 2}
    u(x,y) = \frac{4(1 + y)}{\sqrt{2 \pi}} e^{-8x^2}.
\end{equation*}
The source term $S$ and the boundary data $f$ and $g$ are defined accordingly:
\begin{equation*}
\begin{cases}
    S(x,y) = \frac{4(1 + y)}{\sqrt{2 \pi}} e^{-8x^2} \left( 256 x^2 - 15 + y^2) \right),\\
    f(x) = g(x) = \frac{4}{\sqrt{2 \pi}} e^{-8x^2}.
    \end{cases}
\end{equation*}

In both cases, we consider a Gaussian convolution kernel $\varphi_\alpha$ i.e.
\begin{equation}
\label{def conv kernel}
    \varphi_\alpha(x,y)=\frac{1}{\alpha^2 2 \pi}
e^{-\frac{x^2+y^2}{2 \alpha^2}},
\end{equation}
which satisfies the Levy-kernel condition of Lemma \ref{from-alibaud-et-al} with $s=2$.

\subsection*{Discretization setting}

For the discretization of the system \eqref{pt12d}-\eqref{pt22d}-\eqref{pt32d}, we use a finite difference method of order $2$ described as follows.\\

We first define the uniform grid $\Gamma$ on the bounded domain~$[a,b] \times [0,1]$:
$$
\Gamma_j^n = (x_j,y_n) \quad \text{with} \quad 
\begin{cases}
x_j = a + (j-1)\Delta_x,   & j = 1,...,n_x \\
y_n = (n-1)\Delta_y,   & n=1,...,n_y,
\end{cases}
$$
where $\Delta_x$ and $\Delta_y$ are the discretization steps given by
$$
\Delta_x = (b-a)/(n_x-1), \quad \Delta_y = 1/(n_y-1).
$$
We then approximate the second derivatives $u_{xx}$ and $u_{yy}$ by means of the five-point stencil finite-difference scheme and~$u_y$ using a central finite difference and derive the discrete system:
\begin{eqnarray}
\frac{u_{j}^{n+1} - 2 u_j^n + u_{j}^{n-1}}{\Delta_y^2} + \frac{u_{j+1}^{n} - 2 u_j^n + u_{j-1}^{n}}{\Delta_x^2} + k^2\eta_j^n u_j^n
&=&S_j^n \label{e1}\\
\frac{u_{j}^2 - u_{j}^0}{2 \Delta_y}
&=&f_j \label{e2}\\
u_{j}^1&=&g_j.\label{e3}
\end{eqnarray}
where 
$$
u_j^n \approx u(x_j,y_n), \quad \eta_j^n := \eta(x_j,y_n),  \quad S_j^n := S(x_j,y_n), \quad f_j := f(x_j), \quad \text{and} \quad g_j := g(x_j).
$$
Notice that in \eqref{e2}, by using the central difference to approximate $u_y(x_j,0)$, we define additional nodes $\Gamma_j^0 = (x_j,-\Delta_y)$ lying outside the initial domain $[a,b]\times [0,1]$ and consequently have additional unknowns $u_j^0 \approx u(x_j,-\Delta_y)$.

In equations \eqref{e2} and \eqref{e3}, the index $j$ runs from $1$ to $n_x$, while in \eqref{e1}, $j$ ranges from~$2$ to~$n_x-1$ and~$n$ ranges from~$1$ to~$n_y-1$.
At the boundary node along $x$-direction (i.e. $j=1$ and $j=n_x$), $u_{xx}(x_j,y_n)$ is approximated by the second order scheme:
\begin{equation}
    \label{approx second der edge nodes}
    \begin{cases}
    u_{xx}(x_1,y_n) \approx (2u_{1}^{n} - 5 u_2^n + 4u_{3}^{n} - u_4^n )/\Delta_x^2, \\
    u_{xx}(x_{n_x},y_n) \approx (2u_{n_x}^{n} - 5 u_{n_x-1}^n + 4u_{n_x-2}^{n} - u_{n_x-3}^n)/\Delta_x^2.
    \end{cases}
\end{equation}
Hence at the boundary nodes $j=1$ and $j=n_x$, equation \eqref{e1} is replaced respectively by
\begin{eqnarray}
\frac{u_{1}^{n+1} - 2 u_1^n + u_{1}^{n-1}}{\Delta_y^2} + \frac{2u_{1}^{n} - 5 u_2^n + 4u_{3}^{n} - u_4^n }{\Delta_x^2} + k^2\eta_1^n u_j^n
&=&S_1^n,\\
\frac{u_{n_x}^{n+1} - 2 u_{n_x}^n + u_{n_x}^{n-1}}{\Delta_y^2} + \frac{2u_{n_x}^{n} - 5 u_{n_x-1}^n + 4u_{n_x-2}^{n} - u_{n_x-3}^n }{\Delta_x^2} + k^2\eta_{n_x}^n u_j^n
&=&S_{n_x}^n.
\end{eqnarray}
    
In summary, we obtain the following iterative scheme:
\begin{eqnarray}
u_{j}^1&=& g_j\label{ds1}\\
u_{j}^2 - u_{j}^0 &=&  2 \Delta_y f_j\label{ds2}\\
u_{1}^{n+1} -  (\Lambda_1^n - 4\gamma) u_1^n + \gamma \left(-u_{4}^{n}+4u_{3}^{n} -5u_{2}^{n}\right) + u_1^{n-1} &=&\Delta_y^2 S_1^n\label{ds3}\\
u_{j}^{n+1} - \Lambda_j^n u_j^n + u_j^{n-1} + \gamma \left(u_{j+1}^{n}+u_{j-1}^{n}\right) &=&\Delta_y^2 S_j^n\label{ds4} \\
u_{n_x}^{n+1} -  (\Lambda_{n_x}^n - 4\gamma) u_{n_x}^n  + \gamma \left(-u_{n_x-3}^{n}+4u_{n_x-2}^{n} -5u_{n_x-1}^{n}\right)+ u_{n_x}^{n-1} &=&\Delta_y^2 S_{n_x}^n\label{ds5}
\end{eqnarray}
where 
$$
\Lambda_j^n = 2+ 2 \gamma - k^2  \Delta_y^2 \eta_j^n \quad \text{and} \quad \gamma = \Delta_y^2/\Delta_x^2.
$$
In \eqref{ds3},\eqref{ds4},\eqref{ds5}, the index $n$ runs from $1$ to $n_y-1$.
In \eqref{ds4}, the index $j$ runs from $2$ to $n_x-1$.

By defining the column vector
$$
U^n = (u_1^n, u_2^n,...,u_{n_x}^n)^\top, n=0,1,...,n_y.
$$

we can rewrite the discrete system \eqref{ds1}-\eqref{ds5} in the matrix form:
\begin{eqnarray}
U^1 &=& G \label{tt1}\\
U^2 - U^0 &=&2\Delta_z F \label{tt2} \\
U^{2} - A_1 U^1 + U^{0} &=& \Delta_y^2 S^1,\label{tt3} \\
U^{n+1} - A_n U^n + U^{n-1} &=& \Delta_y^2 S^n, n=2,...,n_y-1. \label{tt4}
\end{eqnarray}
where 
$$
G := 
\begin{pmatrix}
g(x_1) \\
g(x_2) \\
\vdots \\
\vdots\\
g(x_{n_x})
\end{pmatrix}
, \quad
F := 
\begin{pmatrix}
f(x_1) \\
f(x_2) \\
\vdots \\
\vdots\\
f(x_{n_x})
\end{pmatrix}
, \quad
S^n := 
\begin{pmatrix}
S(x_1,y_n) \\
S(x_2,y_n) \\
\vdots \\
\vdots\\
S(x_{n_x},y_n)
\end{pmatrix}
, 
$$
and $A_n$ is the nearly tridiagonal matrix defined by
\begin{equation}
\label{def matrix A_n}
A_n = 
\begin{pmatrix}
-\Lambda_1^n + 4 \gamma & -5 \gamma & 4 \gamma & -\gamma & 0   		& \cdots & 0 \\
\gamma  & -\Lambda_2^n &  \gamma  & 0 & 0 & \cdots &  \vdots  \\
0 & \gamma  & -\Lambda_3^n &  \gamma  & 0 & \ddots &\vdots  \\
0  & 0 &  \ddots        & \ddots       &  \ddots       &  \ddots & 0\\
\vdots &  \ddots  & \ddots &  \ddots        & \ddots       &  \ddots       &  0\\
   \vdots  & \cdots &  0 & 0 &     \gamma   &  -\Lambda_{n_x-1}^n  &   \gamma \\
  0 &        \cdots  &   0  &- \gamma& 4 \gamma&  -5 \gamma & -\Lambda_{n_x}^n + 4 \gamma
\end{pmatrix}.
\end{equation}
From \eqref{tt2} and \eqref{tt3}, we can get rid of the additional unknown vector $U^0$ and get the system
\begin{equation}
\label{final equation}
   \begin{cases}
\qquad \qquad \qquad \qquad \,\,   U^1   = G \\
\qquad \qquad 2 U^{2} + A_1 U^1 \,\,= \Delta_y^2 S^1 + 2 \Delta_y F,\\
U^{n+1} + A_n U^n + U^{n-1} = \Delta_y^2 S^n, \qquad n=2,...,n_y-1.
   \end{cases}
\end{equation}

In order to model the noise in the measured data~$f$ and~$g$,
we consider the noisy versions $G_\epsilon$ and $F_\epsilon$ of the vectors~$G$ and~$F$ defined by 
\begin{equation}
    \label{def noise G and F}
    G_\epsilon = G + \epsilon \vartheta, \quad \text{and} \quad F_\epsilon = F + \epsilon \vartheta,
\end{equation}
where $\vartheta$ is a $n_x$-column vector of zero mean drawn using the normal
distribution.

From \eqref{final equation}, we can  rewrite our discrete system into a single matrix equation: 
$$
A U = B_\epsilon,
$$
where $U$, $B_\epsilon$ are $n_x n_y$-column vectors and $A$ is the $n_x n_y \times n_x n_y$ block-triangular matrix respectively defined by
$$
U= \begin{pmatrix}
U^1 \\
U^2\\
U^3 \\
\vdots\\
\vdots \\
U^{n_y-2} \\
U^{n_y-1} \\
U^{n_y} 
\end{pmatrix}, \,\,
B_\epsilon= \begin{pmatrix}
G_\epsilon \\
\Delta_y^2 S^1 + 2 \Delta_y F_\epsilon\\
\Delta_y^2 S^2 \\
\Delta_y^2 S^3 \\
\vdots \\
\vdots \\
\Delta_y^2 S^{n_y-2} \\
\Delta_y^2 S^{n_y-1} 
\end{pmatrix}, \,\,
$$
and 
$$
A = 
\begin{pmatrix}
I_{n_x} & 0 & \cdots & \cdots& \cdots & \cdots& \cdots & 0 \\
A_1 & 2 I_{n_x} & \ddots &  &&& &\vdots\\
I_{n_x} & A_2 & I_{n_x} & \ddots &&&& \vdots\\
0 & I_{n_x} & A_3 & I_{n_x} & \ddots &&& \vdots\\
\vdots & \ddots  & \ddots & \ddots & \ddots &  \ddots && \vdots\\
\vdots & &\ddots  & \ddots & \ddots & \ddots & \ddots &\vdots\\
\vdots &  & &\ddots    & I_{n_x} & A_{n_y-2} & I_{n_x} & 0  \\
0 &\cdots & \cdots & \cdots & 0  & I_{n_x} & A_{n_y-1} & I_{n_x}
\end{pmatrix},
$$
where $I_{n_x}$ is the square identity matrix of size $n_x$ and the matrices $A_n$ are the sub-matrices defined in \eqref{def matrix A_n}.

The regularized solution $U_\alpha^\epsilon$ is defined as the solution of the minimization problem
\begin{eqnarray}
\label{def U_beta xx}
U_\alpha^\epsilon & = & \argmin_{U \in \Re^{n_x n_y}} \nonumber\\ 
&& \left(\norm{A U - B_\epsilon }^2 + \norm{(I-C_\alpha)E u}^2 + \norm{D_x (I-C_\alpha) E U}^2\right.\nonumber\\
&& +\norm{D_y (I-C_\alpha)E U}^2 + \norm{D_{xx} (I-C_\alpha)E U}^2\\
&& \left.+\norm{D_{yy} (I-C_\alpha)EU}^2 + 2\norm{D_{xy} (I-C_\alpha)E U}^2\right),\nonumber
\end{eqnarray}
where $D_x,D_y,D_{xx}$, $D_{yy},D_{xy}$ are discrete versions of the partial differential operators $\partial_x, \partial_y$, $\partial_{xx}, \partial_{yy},\partial_{xy}$, $C_\alpha$ is the matrix approximating the convolution with the function $\varphi_{\alpha}$ defined in \eqref{def conv kernel} and $E$ is the matrix modeling the extension operator. From \eqref{def U_beta xx}, we compute $U_\alpha^\epsilon $ as the solution of the matrix equation
\begin{equation}
\label{def U_beta}
\left[ A^\top A + E^\top(I_{n_xn_y} - C_\alpha)^\top D (I_{n_xn_y} - C_\alpha)E \right] U_\alpha^\epsilon  = A^\top B_\epsilon,
\end{equation}
where $\alpha$ is the regularization parameter, $I_{n_x n_y}$ is the square identity matrix of size $n_x n_y$ and $D$ is the matrix defined by
$$
D = I_{n_xn_y} + D_x^\top D_x +D_y^\top D_y + D_{xx}^\top D_{xx} + D_{yy}^\top D_{yy} + 2 D_{xy}^\top D_{xy}.
$$

\subsection*{Selection of the regularization parameter}

The choice of the regularization parameter $\alpha$ is a crucial step of the regularization. Indeed, the reconstruction error $U - U_{\alpha}^\epsilon $ has two components: the regularization error $U - U_{\alpha}$ (corresponding to exact data) and the data error propagation $ U_\alpha - U_{\alpha}^\epsilon$. The former error is generally monotonically increasing with respect to $\alpha$ and attains its minimum at $\alpha=0$ while the latter error blows up as $\alpha$ goes to $0$ and decreases when $\alpha$ gets larger. Consequently, the reconstruction error norm $||U - U_{\alpha}^\epsilon||$ is minimal in some located region (depending on the noise level $\epsilon$ in the data) where both error terms have approximately the same magnitude. Outside that region, the reconstruction error is dominated by one of the two error terms which leads to an undesirable approximate solution $U_{\alpha}^\epsilon$.



In the following, we consider the heuristic selection
rule~\eqref{discrete alpha quasi opt}-\eqref{discrete quasi-optimality}
which has a similitude with the discrete quasi-optimality rule
\cite{quasi_opt_3,quasi_opt_1,recent_result_quasi_opt,quasi_opt_2} except for
the denominator which in our case is not equal to one.

Let $(\alpha_n)_n$ be a sample of the regularization parameter $\alpha$ on a discrete grid defined as 
\begin{equation}
\label{discrete alpha quasi opt}
\alpha_n := \alpha_0 q^n,\quad \alpha_0 \in (0,\norm{T}^2], \quad 0< q< 1, \quad n =1,...,N_0,
\end{equation}
we consider the parameter $\alpha(\epsilon)^*$ defined by
\begin{equation}
\label{discrete quasi-optimality}
\alpha(\epsilon)^* = \alpha_{n_*}, \quad \text{with} \quad n_* = \argmin_{n \in \{1,...,N_0\}} \frac{\norm{U_{\alpha_{n}}^\epsilon - U_{\alpha_{n+1}}^\epsilon  }}{\alpha_{n}- \alpha_{n+1}}.
\end{equation}

The heuristic behind the rule \eqref{discrete quasi-optimality} is the following:

Indeed, we aim at approximating the best regularization parameter $\alpha(\epsilon)$ (over the chosen grid) which minimizes the reconstruction error norm $||U - U_{\alpha}^\epsilon||$ over the grid $(\alpha_n)_n$, i.e.
$$
\alpha(\epsilon) = \alpha_{n(opt)} \quad \text{with} \quad n(opt) = \argmin_{n}\,\, K(\alpha_n) : = \norm{U - U_{\alpha_n}^\epsilon}.
$$

Given that minimizers of a differentiable function are critical points of that function, provided the function $K$ is differentiable, $\alpha_{n(opt)}$ can be characterized as a minimizer of the absolute value of the derivative of function $K$, that is 
$$
n(opt) \approx \argmin_{n}\,\, \abs{K'(\alpha_n)}.
$$
By approximating the derivative $K'(\alpha_n)$ of the function $K$ at $\alpha_n$ by its growth rates over the grid $(\alpha_n)_n$, we get that
$$
n(opt) \approx \argmin_{n}\,\,\,\,\,
\abs{ \frac{K(\alpha_{n+1}) - K(\alpha_{n})}{\alpha_{n+1} - \alpha_n} }.
$$
However, since the exact solution~$u$ is unknown, we cannot evaluate the function~$K$.
In such a setting, we search a tight upper bound of the function~$K$ and aim at
minimizing that upper bound. Using the triangle inequality, we have
\begin{equation}
\label{ estimate function K}
\abs{ K(\alpha_{n+1}) - K(\alpha_{n})} =
\abs{ \norm{U -U_{\alpha_{n+1}}^\epsilon}  - \norm{U -U_{\alpha_{n}}^\epsilon} }
\leq \norm{U_{\alpha_{n+1}}^\epsilon - U_{\alpha_{n}}^\epsilon}
\end{equation}

Hence from \eqref{ estimate function K}, we get an upper bound of the unknown term $\abs{ K(\alpha_{n+1}) - K(\alpha_{n+1})}$ which is actually computable.

By approximating $\abs{ K(\alpha_{n+1}) - K(\alpha_{n+1})}$ by its upper bound in \eqref{ estimate function K}, we get
$$
n(opt) \approx \argmin_{n}\,\,\,\,\,
\abs{ \frac{\norm{U_{\alpha_{n+1}}^\epsilon - U_{\alpha_{n}}^\epsilon}}{\alpha_{n+1} - \alpha_n} } = \frac{\norm{U_{\alpha_{n+1}}^\epsilon - U_{\alpha_{n}}^\epsilon}}{\alpha_{n} - \alpha_{n+1}} .
$$
which is precisely the definition of our heuristic selection rule \eqref{discrete quasi-optimality}.

To illustrate the efficiency of the selection rule \eqref{discrete quasi-optimality}, on Figure \ref{Fig heur rule exp 1} (resp. Figure \ref{Fig heur rule exp 2}), we exhibit the curve of the reconstruction error along with the selected parameter $\alpha(\epsilon)^*$ for each noise level for Example 1 (resp. Example 2).

\begin{figure}[h!]
\begin{center}
\includegraphics[scale=0.5]{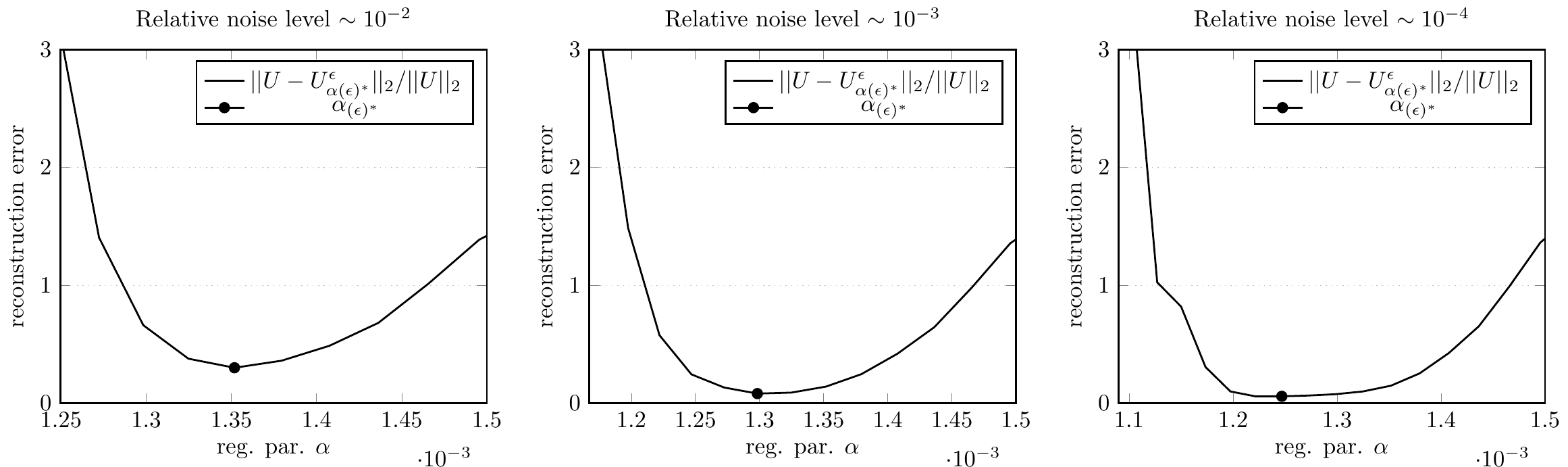} 
\end{center}
\caption{Reconstruction error versus regularization parameter $\alpha$ along with selected regularization parameter $\alpha(\epsilon)^*$ for Example 1.}
\label{Fig heur rule exp 1}
\end{figure}

\begin{figure}[h!]
\begin{center}
\includegraphics[scale=0.5]{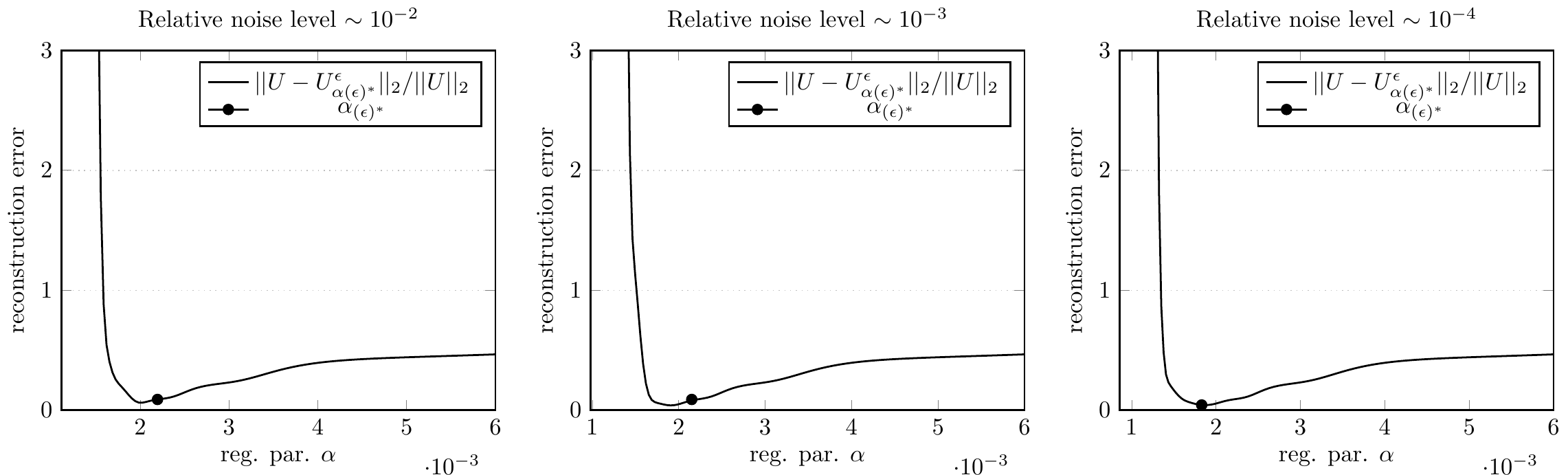} 
\end{center}
\caption{Reconstruction error versus regularization parameter $\alpha$ along with selected regularization parameter $\alpha(\epsilon)^*$ for Example 2.}
\label{Fig heur rule exp 2}
\end{figure}

\subsection*{Results and comments}
In the simulations, we consider three noise
levels $\epsilon_i$ $(i=2,3,4)$
such that the relative error in the data ($Red$) satisfies
$$
Red = \frac{\norm{F-F_{\epsilon_i}}_2}{\norm{F}_2} \approx \frac{\norm{G-G_{\epsilon_i}}_2}{\norm{G}_2} \approx 10^{-i}.
$$
 We choose \texttt{Matlab} as the coding environment and we solve equation \eqref{def U_beta} using a generalised minimal residual method (GMRES) with ortho-normalization based on Householder reflection. We choose as initial guess the solution from the \texttt{Matlab} direct solver \texttt{lmdivide}.
 
Figures~\ref{Figure comp_u_u_beta exple 1} (resp. \ref{Figure comp_u_u_beta exple 2}) compares the exact solution $u$ to the reconstruction $u_{\alpha(\epsilon)^*}^\epsilon$ for each noise level for Example 1 (resp. Example 2). From these Figures, we observe that the reconstruction gets better as the noise level decreases.

On Figures~\ref{Figure comp_u_u_beta_z_075 exple 1} and \ref{Figure comp_u_u_beta_z_1 exple 1} (resp. \ref{Figure comp_u_u_beta_z_075 exple 2} and \ref{Figure comp_u_u_beta_z_1 exple 2}), we compare the exact function~$u$
and the regularized solution $u_{\alpha(\epsilon)^*}^\epsilon$ at $y=0.75$ and $y=1$ for Example 1 (resp. Example 2) for each noise level. 

Table~\ref{Table rel error fct z exple 1} (resp. \ref{Table rel error fct z exple 2}) presents the numerical values of the relative errors 
$$
\frac{\norm{u(\cdot,y) - u_{\alpha(\epsilon)^*}^\epsilon(\cdot,y)}_2}{\norm{u(\cdot,y)}_2}
$$
for $y=0,0.25,0.5,0.75$ and $ y=1$ for Example 1( resp. Example 2) for each noise level. From these Tables, we observe  that the reconstruction error get smaller as $y$ approaches $0$ and as the noise level decreases.

From Figures \ref{Figure comp_u_u_beta exple 1} to \ref{Figure comp_u_u_beta_z_1 exple 2} and Tables \ref{Table rel error fct z exple 1} and \ref{Table rel error fct z exple 2}, we can see that our mollifier regularization approach yields quite good results. Moreover, as predictable, the reconstruction gets better when the noise level decreases and when we get closer to the boundary side where boundary data are given.

\paragraph{Acknowledgement}
The authors are grateful to N. Alibaud for interesting comments
and discussions during the development of the proposed methodology.
They also wish to thank T. Le Minh, for nice and fruitful exchanges.

\begin{figure}[h!]
\begin{center}
\includegraphics[scale=0.5]{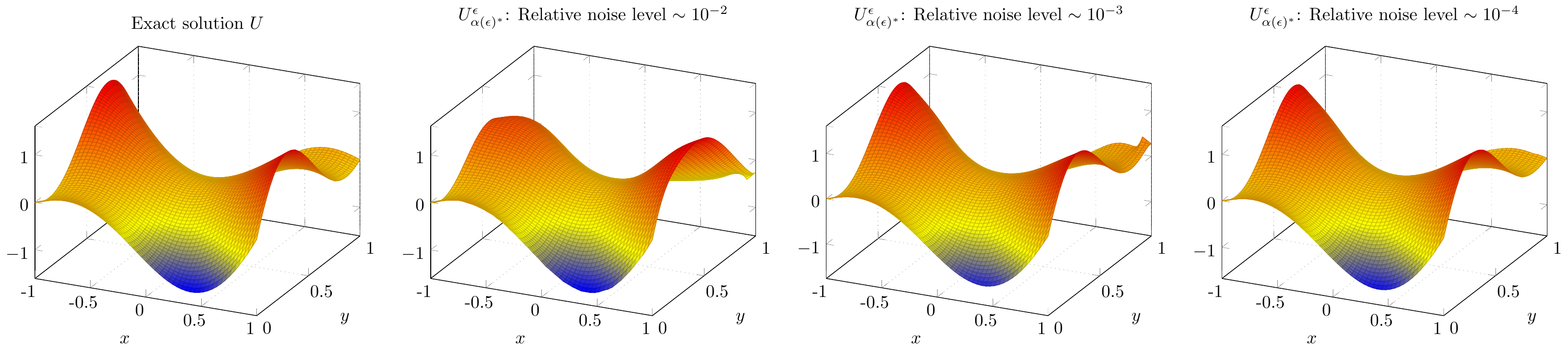} 
\end{center}
\caption{Comparison of the exact solution $u$ of Example 1  and the regularized solution  $u_{\alpha(\epsilon)^*}^\epsilon$ for each noise level.}
\label{Figure comp_u_u_beta exple 1}
\end{figure}

\begin{figure}[h!]
\begin{center}
\includegraphics[scale=0.65]{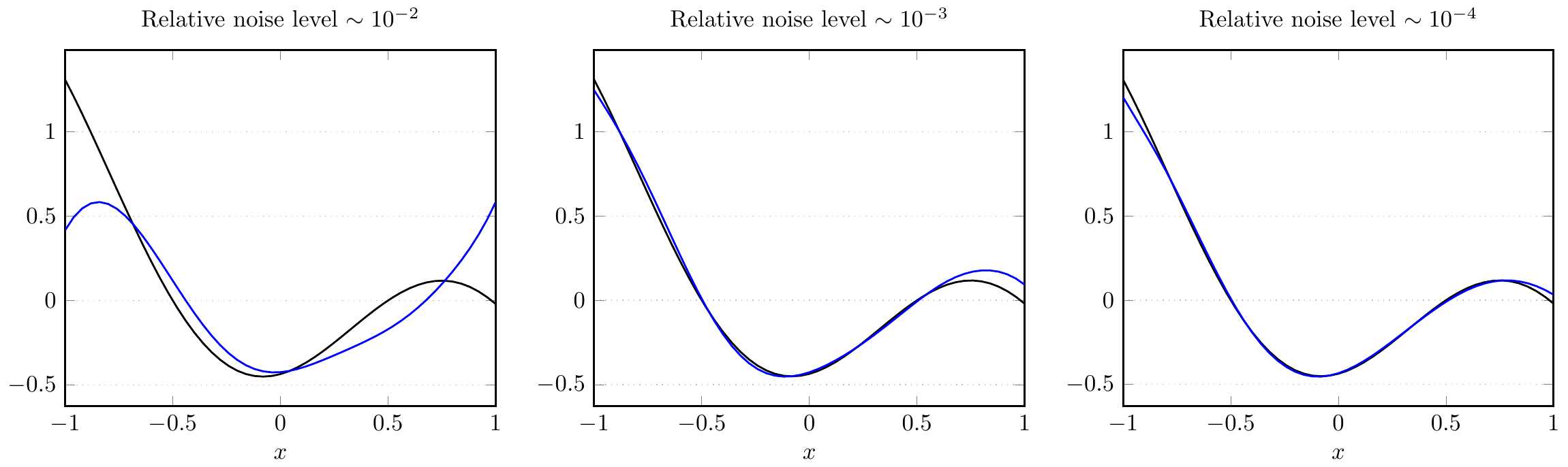} 
\end{center}
\caption{Comparison of the exact solution $u(\cdot,y)$ (black curve) of Example 1 and the regularized solution  $u_{\alpha(\epsilon)^*}^\epsilon(\cdot,y)$ at $y=0.75$ (blue curve) for each noise level.}
\label{Figure comp_u_u_beta_z_075 exple 1}
\end{figure}

\begin{figure}[h!]
\begin{center}
\includegraphics[scale=0.65]{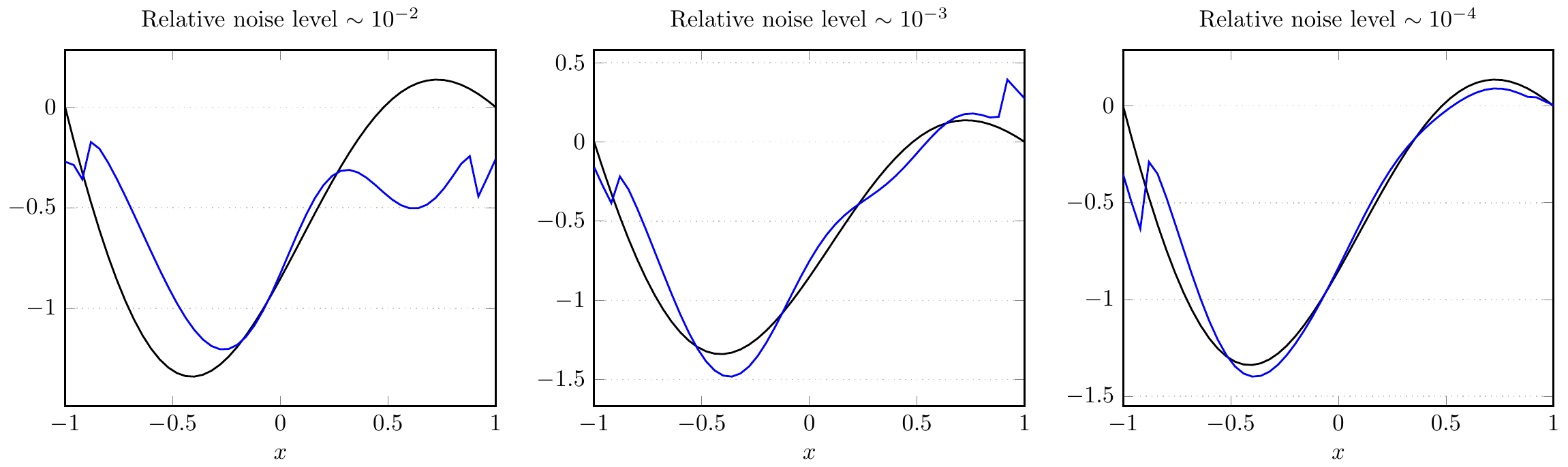} 
\end{center}
\caption{Comparison of the exact solution $u(\cdot,y)$ (black curve) of Example 1 and the regularized solution $u_{\alpha(\epsilon)^*}^\epsilon(\cdot,y)$ (blue curve) at $y=1$ for each noise level.}
\label{Figure comp_u_u_beta_z_1 exple 1}
\end{figure}

\begin{table}[h!]
\begin{center}
\includegraphics[scale=1.5]{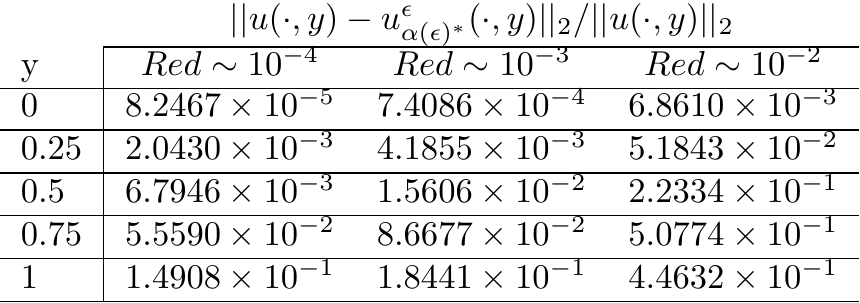} 
\end{center}
\caption{Relative $L^2$ error between the exact solution $u(\cdot,y)$ of Example 1  and the regularized solution $u_{\alpha(\epsilon)^*}^\epsilon(\cdot,y)$ for $y=0,0.25,0.5,0.75,1$.}
\label{Table rel error fct z exple 1}
\end{table}

\begin{figure}[h!]
\begin{center}
\includegraphics[scale=0.5]{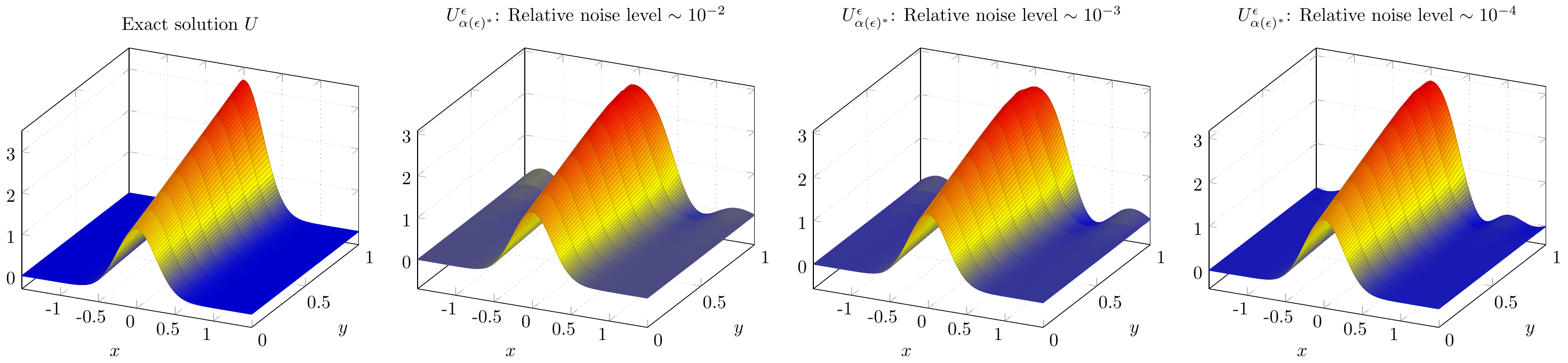} 
\end{center}
\caption{Comparison of the exact solution $u$ of Example 2 and the regularized solution  $u_{\alpha(\epsilon)^*}^\epsilon$ for each noise level.}
\label{Figure comp_u_u_beta exple 2}
\end{figure}

\begin{figure}[h!]
\begin{center}
\includegraphics[scale=0.65]{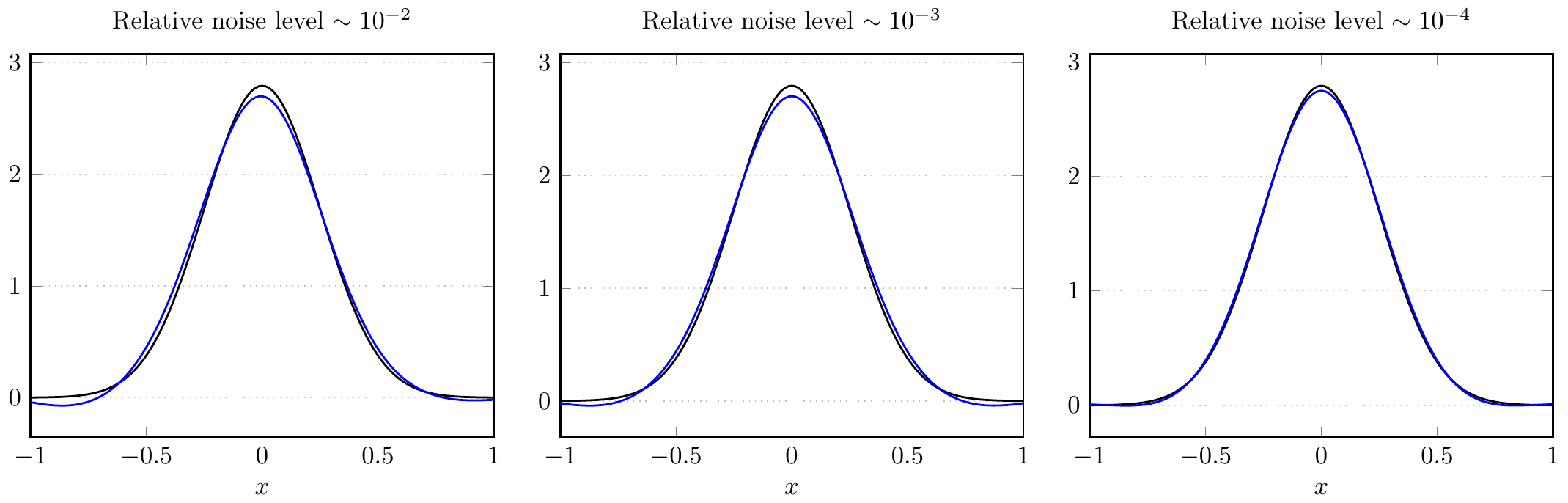} 
\end{center}
\caption{Comparison of the exact solution $u(\cdot,y)$ (black curve) of Example 2 and the regularized solution  $u_{\alpha(\epsilon)^*}^\epsilon(\cdot,y)$ at $y=0.75$ (blue curve) for each noise level.}
\label{Figure comp_u_u_beta_z_075 exple 2}
\end{figure}

\begin{figure}[h!]
\begin{center}
\includegraphics[scale=0.65]{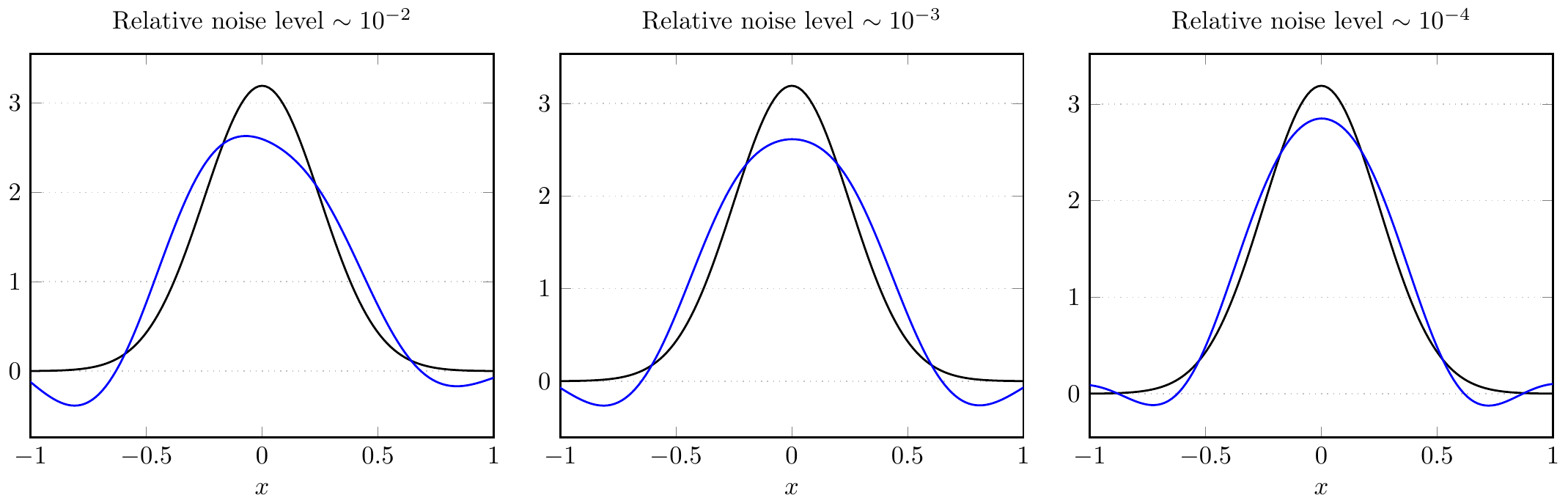} 
\end{center}
\caption{Comparison of the exact solution $u(\cdot,y)$ (black curve) of Example 2 and the regularized solution  $u_{\alpha(\epsilon)^*}^\epsilon(\cdot,y)$ (blue curve) at $y=1$ for each noise level.}
\label{Figure comp_u_u_beta_z_1 exple 2}
\end{figure}

\begin{table}[h!]
\begin{center}
\includegraphics[scale=1.5]{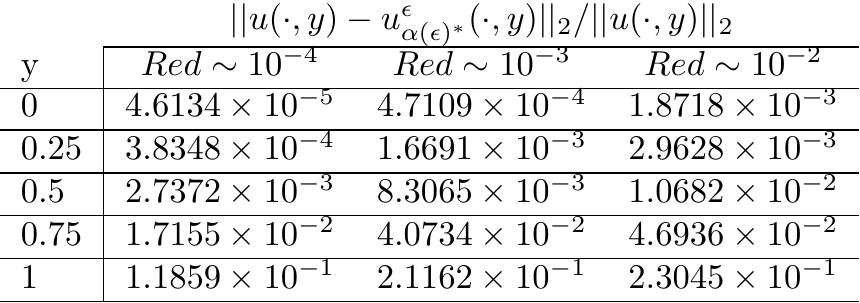} 
\end{center}
\caption{Relative $L^2$ error between the exact
solution $u(\cdot,y)$ of Example 2 and the regularized solution
$u_{\alpha(\epsilon)^*}^\epsilon(\cdot,y)$ for $y=0,0.25,0.5,0.75,1$.}
\label{Table rel error fct z exple 2}
\end{table}

\clearpage
\bibliographystyle{abbrv}
\bibliography{Helmoltz}

\end{document}